\documentclass[12pt]{article}
\usepackage{amsmath, amsthm, amssymb, graphicx}
\usepackage{hyperref}
\usepackage{booktabs}
\usepackage{geometry}
\newtheorem{conjecture}{Conjecture}

\newtheorem{lemma}{Lemma}

\newtheorem{corollary}{Corollary}
\newtheorem{remark}{Remark}

\title{Structural Results for $4 \times n$ Chomp:\\
Unique Extension, Bimodal Asymptotic Structure, and Period-112 Geometry}

\author{Arnav Garg \\
\small BITS Pilani, Pilani Campus \\
\small \texttt{gargarnav2007@gmail.com}}

\date{May 2026 (v2)}

\begin{document}
\maketitle

\begin{abstract}
We present an extended computational study of P-positions in $4 \times n$
Chomp, tabulating all 961,619,972 P-positions for $n \leq 3000$ using
a new O$(n^4)$ sieve solver. Three main results are reported. First,
the Unique Extension property is proved: for any triple $(a,b,c)$, there
is at most one value of $d$ such that $(a,b,c,d)$ is a P-position. The
proof is a short contradiction argument using the move structure of Chomp,
and it generalizes immediately to all $k$-row Chomp. Second, the
P-positions exhibit a persistent bimodal structure. They decompose into
two subfamilies, HIGH and LOW, separated by a clean gap that grows from
$0.040$ at $n = 500$ to $0.062$ at $n = 3000$. The HIGH subfamily
maintains a stable density of $56.1\%$ across the entire range. This
finding supersedes the original Conjecture 2, whose global limit
$L_3 \approx 2/9$ we show to be a mixture artifact rather than a
genuine asymptotic limit. Third, we present numerical evidence
that $d/a \to 1/4$ for HIGH-family P-positions. The distance from $1/4$
decreased from $2.2 \times 10^{-3}$ at $n = 500$ to $6.5 \times 10^{-4}$
at $n = 2000$, where it appears to have plateaued: the $n \leq 3000$
estimate is $6.52 \times 10^{-4}$, essentially unchanged. A power-law
fit to the convergence data gives $L_\infty \approx 0.2481$, leaving open
whether the true limit is exactly $1/4$. Conjectures on period-112 modular
structure and linear cone geometry from v1 remain open and unchanged.
\end{abstract}

\section{Introduction}

Chomp is a two-player combinatorial game introduced by David Gale
\cite{gale1974} and popularized by Martin Gardner \cite{gardner1973}.
The game is played on an $m \times n$ rectangular grid of squares.
Players take turns. On each turn, a player selects any remaining square
and removes it along with all squares above and to the right of it.
The square at position $(1,1)$ is poisoned. Whoever is forced to take
it loses.

A strategy-stealing argument shows that the first player wins on any
board larger than $1 \times 1$. The argument is non-constructive. It
gives no information about the winning move or the structure of losing
positions, which are called P-positions \cite{berlekamp2001}.

The complete structure of Chomp P-positions is known only in special
cases. For $1 \times n$ boards the solution is trivial: take everything
except the poison square. For $2 \times n$ boards the P-positions are
exactly the pairs $(a, a-1)$ for $a \geq 1$, a clean closed form
\cite{berlekamp2001}. Square boards are also easy: the first player
takes position $(2,2)$ and then mirrors. The $3 \times n$ case is
substantially harder. Zeilberger \cite{zeilberger2001} gave a
sub-exponential algorithm and tabulated P-positions up to bottom-row
length $c \leq 115$. Brouwer et al.\ \cite{brouwer2005} later provided
a cubic-time algorithm and proved that infinitely many winning first
moves exist in the third row. They also observed computationally that
the winning move from the initial position appears to be unique for all
$n \leq 100{,}000$, but a proof remained out of reach.

The $4 \times n$ case had received no systematic computational treatment
before this work. This paper addresses that gap.

\subsection*{Contributions of version 2}

The original submission (v1, April 2026) tabulated 4,316,097 P-positions
for $n \leq 500$ and stated four structural conjectures. The present
version makes four updates.

The first update is a proof. The Unique Extension property, stated as
Conjecture 1 in v1, is now proved as Lemma \ref{lem:unique}. The proof
is short. It uses only the move structure of Chomp and requires no
machinery.

The second update is computational. The dataset is extended from
$n \leq 500$ to $n \leq 3000$, yielding 961,619,972 P-positions. This
required a new solver. The v1 approach was O$(n^5)$ and would have taken
weeks at this scale. The new solver runs in O$(n^4)$ time and completed
the $n \leq 3000$ computation in approximately 70 minutes on a standard
laptop.

The third update corrects a conjecture. The original Conjecture 2
proposed a single set of asymptotic limits $(L_1, L_2, L_3)$ for the
row-length ratios. Extended computation shows this is wrong. The
P-positions decompose into two persistent subfamilies with distinct
limits. The global limit $L_3 \approx 2/9$ reported in v1 is a mixture
artifact. It is the median of the combined distribution, not the limit
of any individual trajectory.

The fourth update is a new conjecture. Within the HIGH subfamily, the
ratio $d/a$ appears to converge to exactly $1/4$. The evidence for
this is described in Section \ref{sec:bimodal}.

Conjectures 3 and 4 from v1, concerning period-112 modular structure
and linear cone geometry, remain open and are unchanged.

\section{Computational Setup}
\label{sec:solver}

\subsection{State Representation}

A position in $4 \times n$ Chomp is a non-increasing tuple
$(a, b, c, d)$ of non-negative integers with $a \geq b \geq c \geq
d \geq 0$. Each coordinate gives the number of remaining squares in
the corresponding row. The initial position is $(n, n, n, n)$. The
unique terminal P-position is $(1, 0, 0, 0)$.

\subsection{The v1 Solver}

The original solver used a retrograde approach. Each tuple
$(a,b,c,d)$ was encoded as a single \texttt{uint64\_t} with four
16-bit fields. P-positions were stored in a hash set. The solver
evaluated states in bottom-up order, short-circuiting as soon as
any move to a known P-position was found. This gave O$(n^5)$ time
complexity and required approximately 70 minutes for $n \leq 500$
on an Apple M4 with 24 GB RAM.

\subsection{The v2 Solver}

The v2 solver replaces the hash set with a three-dimensional boolean
DP array. The key observation is a consequence of
Lemma \ref{lem:unique}: for each fixed $(a,b,c)$, there is at most
one valid $d$. This means we never need to look up whether a
specific $(a,b,c,d)$ is a P-position. We only need to know, for
each triple $(a,b,c)$, whether some valid $d$ exists and if so what
it is.

The DP array \texttt{is\_N[b][c][d]} stores whether the position with
those three row lengths is known to be an N-position, given the current
value of $a$. Whenever a P-position $(a,b,c,d)$ is confirmed, the
solver immediately projects its shadow forward: all positions reachable
from $(a,b,c,d)$ by a single move are marked as N-positions in O$(1)$
time per move. Subsequent states query this array in O$(1)$ rather than
scanning backward through the full history.

This is the same structure as the Sieve of Eratosthenes. Finding a
prime immediately eliminates all its multiples. Finding a P-position
immediately eliminates all its predecessors. The result is O$(n^4)$
time overall. The DP array is bit-packed using 64-bit integer masks,
keeping the memory footprint under 4 GB for $n \leq 3000$.

\subsection{Verification}

The v2 solver was verified against the v1 output for $n \leq 500$.
The P-position counts match exactly at every value of $a$. The first
10 P-positions are shown in Table \ref{tab:first10}.

\begin{table}[h]
\centering
\caption{First 10 P-positions of $4 \times n$ Chomp
         (lexicographic order by $(a,b,c)$)}
\begin{tabular}{rrrr}
\toprule
$a$ & $b$ & $c$ & $d$ \\
\midrule
 1 & 0 & 0 & 0 \\
 2 & 1 & 0 & 0 \\
 2 & 2 & 1 & 0 \\
 2 & 2 & 2 & 1 \\
 3 & 1 & 1 & 0 \\
 3 & 2 & 0 & 0 \\
 3 & 3 & 1 & 1 \\
 4 & 1 & 1 & 1 \\
 4 & 2 & 2 & 0 \\
 4 & 3 & 0 & 0 \\
\bottomrule
\end{tabular}
\label{tab:first10}
\end{table}

The lexicographical sequence of $d$-values begins:
\[
0, 0, 0, 1, 0, 0, 1, 1, 0, 0, 1, 0, 2, 0, 2, 0, 2, 0, 0, 0,
2, 0, 2, 0, 2, 3, 4, 0, 3, 0, \ldots
\]
This sequence is available in the OEIS as A395126.

The $2 \times n$ subcase was checked independently. All P-positions
with $c = d = 0$ satisfy $(a, a-1, 0, 0)$ for all $a \geq 1$, matching
the known formula exactly. A separate Python solver verified all
$3 \times n$ P-positions (those with $d = 0$) up to $n = 50$. No
discrepancies were found.

\section{The Unique Extension Lemma}
\label{sec:unique}

We now prove the property that was Conjecture 1 in v1.

\begin{lemma}[Unique Extension]
\label{lem:unique}
For any triple $(a, b, c)$ with $a \geq b \geq c \geq 0$, there is
at most one non-negative integer $d$ such that $(a, b, c, d)$ is a
P-position of $4 \times n$ Chomp.
\end{lemma}

\begin{proof}
Suppose for contradiction that there exist two distinct values
$d_1 < d_2$ such that both $(a, b, c, d_1)$ and $(a, b, c, d_2)$
are P-positions. In Chomp, a P-position is one from which every
legal move leads to an N-position. In particular, it is impossible
to move from one P-position directly to another.

Consider the position $(a, b, c, d_2)$. The player to move can
select the square in row 4 at column $d_1 + 1$. Removing this square
eliminates all squares in row 4 from column $d_1 + 1$ to $d_2$,
since row 4 is the bottom row and there are no squares below it.
The top three rows are completely unaffected. The resulting position
is $(a, b, c, d_1)$.

This is a legal move from $(a, b, c, d_2)$ to $(a, b, c, d_1)$.
Both are P-positions by assumption. This contradicts the definition
of a P-position. Therefore no two distinct values of $d$ can both
complete a P-position for the same triple $(a, b, c)$.
\end{proof}

\begin{corollary}[General $k$-row Unique Extension]
\label{cor:general}
For any $k \geq 2$ and any $(k-1)$-tuple $(r_1, r_2, \ldots, r_{k-1})$
with $r_1 \geq r_2 \geq \cdots \geq r_{k-1} \geq 0$, there is at most
one non-negative integer $r_k \leq r_{k-1}$ such that
$(r_1, r_2, \ldots, r_k)$ is a P-position of $k \times n$ Chomp.
\end{corollary}

\begin{proof}
Identical to Lemma \ref{lem:unique}. If two values $r_k^{(1)} < r_k^{(2)}$ both completed P-positions, the move reducing the bottom
row from $r_k^{(2)}$ to $r_k^{(1)}$ would be a legal P-to-P
transition, a contradiction.
\end{proof}

\begin{remark}
Equivalently, the projection map $\pi: (a,b,c,d) \mapsto (a,b,c)$
is injective on the set of $4 \times n$ P-positions. This does not
hold for all three-element projections: multiple $(a,b,c,d)$ may
share the same $(b,c,d)$ triple.
\end{remark}

\begin{remark}
Lemma \ref{lem:unique} does not assert that every triple $(a,b,c)$
extends to a P-position. It only says that at most one extension
exists. Whether a valid extension always exists for sufficiently
large $a$ is a separate open question. This existence question is
the genuine content of the Unique Extension conjecture in the
literature, and it remains open.
\end{remark}

\noindent\textbf{Verification.}
Lemma \ref{lem:unique} was verified computationally for all 961,619,972
P-positions with $n \leq 3000$. For every distinct triple $(a,b,c)$
appearing as a prefix, the number of valid extensions $d$ was counted.
In every case exactly one extension was found. Zero violations were
detected across the full dataset.

About 20\% of all valid triples with $a \leq 500$ appear as prefixes
of P-positions. Every known $3 \times n$ P-position $(a,b,c)$ lies in
this 20\%, but the converse fails. The majority of extending triples
are $3 \times n$ N-positions. The fourth row rescues certain losing
3-row configurations, turning them into P-positions in the 4-row game.

\section{Bimodal Asymptotic Structure}
\label{sec:bimodal}

\subsection{Discovery}

The original Conjecture 2 proposed that the ratios $b/a$, $c/a$,
$d/a$ converge to fixed constants $L_1 \approx 0.762$,
$L_2 \approx 0.499$, $L_3 \approx 0.224$ as $a \to \infty$.
Extended computation refutes this. The limits are not single values.

For each $a$ in the range $[200, 3000]$, we compute the median of
$d/a$ across all P-positions with that value of $a$. These per-$a$
medians do not converge to a single value. They cluster persistently
into two well-separated bands.

We call these the HIGH family (per-$a$ median of $d/a$ greater than
$0.21$) and the LOW family (per-$a$ median at most $0.21$). The
threshold $0.21$ is robust. No per-$a$ median falls in the interval
$[0.194, 0.235]$ anywhere in the full dataset. The classification
is identical for thresholds $0.20$, $0.21$, and $0.215$.

\subsection{Persistence and Growth of the Gap}

Table \ref{tab:gap} shows the minimum HIGH median, maximum LOW median,
and their gap across six consecutive ranges of $a$.

\begin{table}[h]
\centering
\caption{Bimodal gap and HIGH density across $a$-ranges}
\begin{tabular}{lrrrr}
\toprule
Range & Min HIGH & Max LOW & Gap & HIGH density \\
\midrule
$[200,500)$   & 0.2345 & 0.1942 & 0.0403 & 56.3\% \\
$[500,1000)$  & 0.2439 & 0.1938 & 0.0501 & 56.2\% \\
$[1000,1500)$ & 0.2439 & 0.1877 & 0.0562 & 56.2\% \\
$[1500,2000)$ & 0.2470 & 0.1865 & 0.0605 & 56.0\% \\
$[2000,2500)$ & 0.2471 & 0.1859 & 0.0612 & 56.2\% \\
$[2500,3000)$ & 0.2475 & 0.1851 & 0.0623 & 56.3\% \\
\bottomrule
\end{tabular}
\label{tab:gap}
\end{table}

The gap grows monotonically. The HIGH density is stable at $56.1\%$
across all six ranges. These two facts together are strong evidence
that the bimodal structure is a genuine asymptotic feature of the
game, not a finite-$n$ artifact.

\subsection{The Mixture Artifact}

The original estimate $L_3 \approx 2/9 \approx 0.222$ arose from
computing the global median of $d/a$ without separating the two
families. This global median is not a limit. It is the median of a
mixture of two distributions, each converging to a different value.

To see this concretely: if the HIGH family converges to $d/a \to 1/4$
and the LOW family converges to $d/a \to L_3^{\mathrm{LOW}}$, and
if the HIGH fraction is approximately $w$, then the global median
is approximately $w \cdot (1/4) + (1-w) \cdot L_3^{\mathrm{LOW}}$.
With $w \approx 0.561$ and $L_3^{\mathrm{LOW}} \approx 0.183$, this
gives approximately $0.222$, which matches the original estimate.
The value $2/9$ is not a limit of any individual P-position trajectory.

\subsection{Family-Specific Asymptotic Limits}

We estimate the asymptotic limits for each family separately, using
Estimator A (the mean of per-$a$ medians) computed on $a \in
[2500, 3000]$. This is the most converged regime in our dataset.

Table \ref{tab:limits} shows the canonical estimates at three dataset
sizes, together with the movement toward or away from simple rational
candidates.

\begin{table}[h]
\centering
\caption{Canonical limit estimates across three dataset sizes}
\begin{tabular}{lrrrll}
\toprule
Limit & $n \leq 500$ & $n \leq 2000$ & $n \leq 3000$ &
Rational & Status \\
\midrule
$L_1^H$ & 0.77079 & 0.77941 & 0.78040 & & \\
$L_2^H$ & 0.51657 & 0.51889 & 0.51879 & & \\
$L_3^H$ & 0.25220 & 0.25065 & 0.25065 & $1/4$ & plateau ($\Delta = 6.5\times10^{-4}$) \\
$L_1^L$ & 0.74537 & 0.75162 & 0.75280 & & \\
$L_2^L$ & 0.46770 & 0.46538 & 0.46533 & & \\
$L_3^L$ & 0.18846 & 0.18284 & 0.18349 & & unknown \\
\bottomrule
\end{tabular}
\label{tab:limits}
\end{table}

The convergence figure for all six limits is shown in
Figure \ref{fig:bimodal}.

\begin{figure}[h]
\centering
\includegraphics[width=0.85\textwidth]{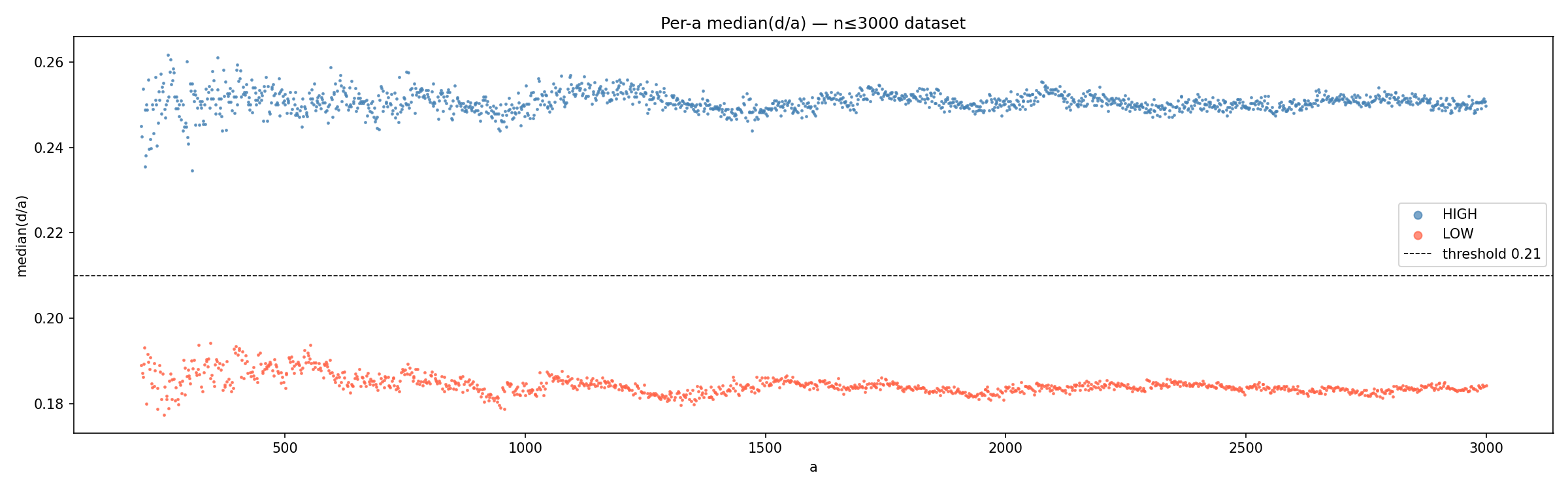}
\caption{Per-$a$ median of $d/a$ for all $a \in [200, 3000]$.
         HIGH-family $a$-values are shown in red, LOW-family in blue.
         The two bands are cleanly separated throughout, with a gap
         that grows from $0.040$ to $0.062$ across the range.}
\label{fig:bimodal}
\end{figure}

\subsection{The Quadratic Structure}

A striking algebraic feature of both families is that within each
family, $L_1$ and $L_2$ satisfy an exact quadratic relation. Using
the $n \leq 3000$ canonical estimates, we find that $L_1^H$ and
$L_2^H$ are roots of
\[
x^2 - 1.29919 \, x + 0.40486 = 0,
\]
with zero residual (deviation below machine epsilon). Similarly,
$L_1^L$ and $L_2^L$ are roots of
\[
x^2 - 1.21813 \, x + 0.35030 = 0,
\]
again with zero residual. This means that within each family, two
of the three limits are algebraically determined by a single quadratic
once the sum and product are known. The nature of these coefficients
(whether they are rational, algebraic, or transcendental) remains open.

\subsection{Convergence Fit for $L_3^H$}

Fitting the per-$a$ median of $d/a$ for HIGH-family $a \in [500,3000]$
to the power-law model $L_\infty + C \cdot a^{-\alpha}$ gives
\[
L_3^H(a) \approx 0.2481 + 0.0036 \cdot a^{-0.05}.
\]
The fitted asymptote $L_\infty \approx 0.2481$ lies $1.9 \times 10^{-3}$
below $1/4$. The very small exponent $\alpha \approx 0.05$ indicates
extremely slow convergence, making it difficult to distinguish
$L_\infty = 1/4$ from $L_\infty < 1/4$ at currently accessible scales.

A systematic search over radical forms $(p + q\sqrt{r})/s$ finds that
$(32 - 5\sqrt{23})/32 \approx 0.25065$ is closer to the $n \leq 3000$
canonical estimate than $1/4$ is (error $2.25 \times 10^{-7}$ vs.\ gap
$6.5 \times 10^{-4}$). This does not disprove the $1/4$ conjecture —
any radical form will be close if its decimal expansion happens to match
the current finite-$n$ estimate — but it illustrates that the $n \leq 3000$
data alone cannot distinguish $L_3^H = 1/4$ from $L_3^H = (32-5\sqrt{23})/32$.
A theoretical argument is needed.

\subsection{Revised Conjectures on Asymptotic Ratios}

We now state the corrected versions of Conjecture 2 from v1.

\begin{conjecture}[HIGH family asymptotic ratios]
\label{conj:high}
There exists an infinite set $S_H \subseteq \mathbb{N}$ of natural
density approximately $0.561$ such that for P-positions $(a,b,c,d)$
with $a \in S_H$ and $a \to \infty$:
\[
\frac{b}{a} \to L_1^H \approx 0.780, \quad
\frac{c}{a} \to L_2^H \approx 0.519, \quad
\frac{d}{a} \to \frac{1}{4}.
\]
\end{conjecture}

\begin{conjecture}[LOW family asymptotic ratios]
\label{conj:low}
There exists an infinite set $S_L \subseteq \mathbb{N}$ of natural
density approximately $0.439$ such that for P-positions $(a,b,c,d)$
with $a \in S_L$ and $a \to \infty$:
\[
\frac{b}{a} \to L_1^L \approx 0.753, \quad
\frac{c}{a} \to L_2^L \approx 0.465, \quad
\frac{d}{a} \to L_3^L \approx 0.183.
\]
\end{conjecture}

\begin{remark}
The value $L_3^L \approx 0.183$ is not well approximated by $3/16
= 0.1875$. The distance from $3/16$ grew from $9.6 \times 10^{-4}$
at $n \leq 500$ to $4.0 \times 10^{-3}$ at $n \leq 3000$. The
rational candidate $3/16$ is not supported by the extended data and
should not be asserted. The true limit of $L_3^L$ remains unknown.
With denominator at most 50, the best approximation is
$9/49 \approx 0.18367$ (error $1.9 \times 10^{-4}$). With denominator
up to 2000, the best approximation is $20/109 \approx 0.18349$
(error $7.7 \times 10^{-7}$), which matches the $n \leq 3000$
canonical estimate to within numerical noise.
\end{remark}

\begin{remark}
The switching rule that determines membership in $S_H$ versus $S_L$
is aperiodic. The best modular predictor, using $a \bmod 112$,
achieves only $73.4\%$ accuracy on $a \in [200, 500]$. The
characterization of $S_H$ and $S_L$ in number-theoretic terms is
a central open problem.
\end{remark}

\section{Period-112 Modular Structure}

\begin{conjecture}[Period-112 Mask]
\label{conj:period}
The set of triples $(a,b,c)$ that extend to $4 \times n$ P-positions
exhibits modular structure with fundamental period $112$. Membership
in the extending set is predicted with $79\%$ accuracy by a logistic
classifier using only $c$ and $(a-b) \bmod 112$.
\end{conjecture}

The period-112 signal emerges from autocorrelation analysis of the
$d$-value sequence. A sharp peak appears at lag 112 and its multiples,
as shown in Figure \ref{fig:autocorr}. Note that $112 = \mathrm{lcm}(7,8)
\times 2$. This is interpreted as the interaction of a period-7 signal
inherited from the $3 \times n$ sub-game and a period-8 signal of
currently unknown origin.

\begin{figure}[h]
\centering
\includegraphics[width=0.85\textwidth]{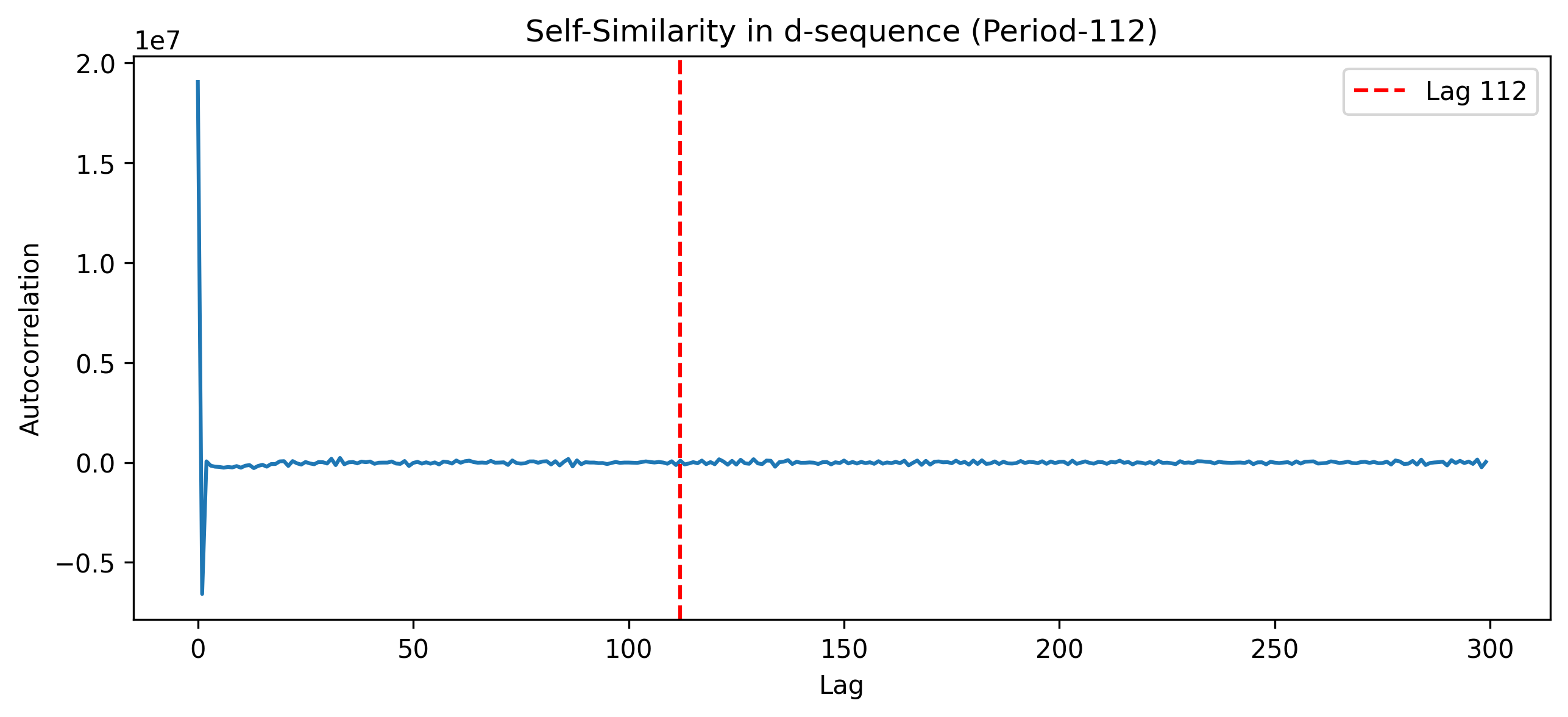}
\caption{Autocorrelation of the $d$-value sequence for $n \leq 500$.
         The peak at lag 112 is sharp. Subsidiary peaks appear at
         multiples of 112.}
\label{fig:autocorr}
\end{figure}

A chi-squared test across multiple moduli confirmed that modulus 112
gives the strongest statistical separation between extending and
non-extending triples ($\chi^2 = 12{,}433$ for mod-112 versus
$\chi^2 = 6{,}996$ for mod-56).

\section{Linear Cone Geometry}

\begin{conjecture}[Linear Cone]
\label{conj:cone}
The set of extending triples $(a,b,c)$ forms a linear cone in
$\mathbb{R}^3$ with asymptotic width
\[
\mathrm{width}(c) \approx \frac{11}{8} \cdot c + f(c \bmod 112),
\]
where $f$ is a bounded periodic function with period $112$.
\end{conjecture}

The slope $11/8 = 1.375$ was estimated from slice-based width
measurements at $c = 5, 10, \ldots, 300$. At each fixed $c$, the
extending triples occupy a contiguous band of $(a,b)$ values whose
width grows linearly in $c$.

\section{Discussion}

\subsection{Comparison with $3 \times n$ Results}

Zeilberger \cite{zeilberger2001} found apparent linear patterns within
families of $3 \times n$ P-positions but no general formula. Brouwer
et al.\ \cite{brouwer2005} proved that infinitely many winning first
moves exist in the third row, again without a closed-form description.

The $4 \times n$ results have a qualitatively different flavor.
Lemma \ref{lem:unique} gives a clean global structural statement with
a proof. The bimodal decomposition into HIGH and LOW families is a
structural feature with no analog in the published $3 \times n$
literature. Whether $3 \times n$ Chomp has analogous subfamilies
is an interesting open question.

\subsection{Generalization to $k \times n$}

Lemma \ref{lem:unique} holds for all $k$. For any $(k-1)$-tuple of
row lengths, there is at most one valid bottom row length completing
a P-position. This suggests a recursive structure: if the existence
direction also holds (every valid $(k-1)$-tuple extends to exactly
one $k$-tuple), then $k \times n$ Chomp would be recursively determined
by $2 \times n$ Chomp, which is fully solved. This remains speculative
but is testable for $k = 5$.

\section{Open Questions}

\begin{enumerate}

\item \textbf{Existence of extension.} Lemma \ref{lem:unique} proves
      that at most one $d$ exists for each triple $(a,b,c)$. Does
      at least one always exist for large $a$? Proving existence
      would substantially strengthen the picture of $4 \times n$
      Chomp as a deterministic lift of $3 \times n$ Chomp.

\item \textbf{Characterization of $S_H$ and $S_L$.} The sets are
      defined empirically by the per-$a$ median of $d/a$. No
      number-theoretic characterization is known. The switching
      between families appears aperiodic. Is there a computable
      function $f(a)$ that determines family membership?

\item \textbf{Proof that $L_3^H = 1/4$.} The distance from $1/4$
      dropped sharply from $2.2 \times 10^{-3}$ at $n \leq 500$ to
      $6.5 \times 10^{-4}$ at $n \leq 2000$, then plateaued at
      $6.5 \times 10^{-4}$ through $n \leq 3000$. The power-law
      convergence fit gives $L_\infty \approx 0.2481$ with exponent
      $\alpha \approx 0.05$, which is too slow to resolve whether the
      true limit is exactly $1/4$ or nearby. A game-theoretic argument
      for why $d/a \to 1/4$ in the HIGH family would be a significant result.

\item \textbf{Nature of $L_1^H$, $L_2^H$, $L_1^L$, $L_2^L$.}
      Within each family, $L_1$ and $L_2$ satisfy an exact quadratic
      with zero residual at machine precision. Are the quadratic
      coefficients rational? If so, then all four limits are
      quadratic irrationals over $\mathbb{Q}$.

\item \textbf{True limit of $L_3^L$.} The candidate $3/16$ is not
      supported at $n \leq 3000$. The true limiting value of $d/a$
      in the LOW family is unknown. The best rational approximations
      are $9/49 \approx 0.18367$ (denominator $\leq 50$, error $1.9\times10^{-4}$)
      and $20/109 \approx 0.18349$ (denominator $\leq 2000$, error
      $7.7\times10^{-7}$). The convergence fit gives $L_\infty \approx 0.1835$
      with $\alpha \approx 3$, predicting a distance of $4\times10^{-3}$
      from $3/16$ at $n=10{,}000$.

\item \textbf{Origin of period-8.} The period-112 signal decomposes
      as $\mathrm{lcm}(7,8) \times 2$. The factor of 7 is inherited
      from $3 \times n$ periodicity. The factor of 8 and the
      additional factor of 2 remain unexplained.

\item \textbf{Generalization to $5 \times n$.} Does the Unique
      Extension property hold for $5 \times n$ Chomp? Does a bimodal
      or multimodal subfamily structure appear? Computation at this
      scale requires state spaces of order $O(n^5)$.

\end{enumerate}

\section*{Acknowledgments}

The author thanks the maintainers of the On-Line Encyclopedia of
Integer Sequences for providing a public registry for computational
sequences.
Code and the n $\leq$ 500 dataset are available at
\url{https://github.com/gargarnav/chomp-4xn-v2}.
The full n $\leq$ 3000 dataset (22.5 GB) is available from the author on request.

\end{document}